\documentclass[12pt]{amsart}
\usepackage{verbatim}

\def\ep{\varepsilon}

\def\wt{\widetilde}
\def\Re{\text{\rm Re}\,}

\newtheorem{thm}{Theorem}

\newtheorem{prop}[thm]{Proposition}

\begin{document}

\baselineskip=15pt

\title[Bergman kernel and logarithmic capacity]
{One dimensional estimates\\
for the Bergman kernel \\ and logarithmic capacity}

\def\nl{\newline\phantom{a}\hskip 7pt}

\author{Zbigniew B\l ocki, W\l odzimierz Zwonek}
\address{\!\!\!\!Uniwersytet Jagiello\'nski \nl Instytut Matematyki
\nl \L ojasiewicza 6 \nl 30-348 Krak\'ow \nl Poland
\nl {\rm Zbigniew.Blocki@im.uj.edu.pl 
\nl Wlodzimierz.Zwonek@im.uj.edu.pl}}

\thanks {The first named author was supported by the Ideas Plus 
grant no. 0001/ID3/2014/63 of the Polish Ministry of Science and 
Higher Education and the second named author by 
the Polish National Science Centre (NCN) Opus grant no.
2015/17/B/ST1/00996}

\begin{abstract}
Carleson showed that the Bergman space for a domain on the plane
is trivial if and only if its complement is polar. Here we give
a quantitative version of this result. One is the Suita conjecture,
established by the first-named author in 2012, the other is 
an upper bound for the Bergman kernel in terms of 
logarithmic capacity. We give some other estimates
for those quantities as well. We also show that the volume of
sublevel sets for the Green function is not convex for
all regular non simply connected domains,
generalizing a recent example of Forn\ae ss.
\end{abstract}

\makeatletter
\@namedef{subjclassname@2010}{%
  \textup{2010} Mathematics Subject Classification}
\makeatother



\maketitle

\section{Introduction}

For $w\in\Omega$, where $\Omega$ is a domain in $\mathbb C$, Carleson
\cite{C} (see also \cite{Co}) showed the Bergman space $A^2(\Omega)$ 
of square integrable holomorphic functions in $\Omega$ is trivial 
if and only if the complement $\mathbb C\setminus\Omega$ is polar. 
The estimate conjectured by Suita \cite{S} and proved in \cite{Binv}
\begin{equation}\label{suita}
     c_\Omega(w)^2\leq\pi K_\Omega(w),\ \ \ w\in\Omega
\end{equation}
gives a quantitative version of one of the implications. Here  
\begin{equation}\label{cap}
c_\Omega(w)=\exp\left(\lim_{z\to w}\big(G_\Omega(z,w)
       -\log|z-w|\big)\right)
\end{equation}
is the logarithmic capacity of $\mathbb C\setminus\Omega$ with
respect to $w$, 
  $$G_\Omega(z,w)=\sup\{u(z)\colon u\in SH(\Omega),\ u<0,\
     \limsup_{\zeta\to w}\left(u(\zeta)-\log|\zeta-w|)<\infty
     \right\}$$ 
is the (negative) Green function,
  $$K_\Omega(w)=\sup\{|f(w)|^2\colon f\in A^2(\Omega),\ 
     ||f||\leq 1\}$$
is the Bergman kernel on the diagonal and $||f||=||f||_{L^2(\Omega)}$.  
 
Our first result is the following upper bound for the Bergman kernel:

\begin{thm} \label{thm1}
Let $\Omega$ be a domain in $\mathbb C$ and $w\in\Omega$.
Assume that $0<r\leq\delta_\Omega(w):=\text{\rm dist\,}
(w,\partial\Omega)$. Then
  $$K_\Omega(w)\leq\frac 1{\displaystyle -2\pi r^2
     \max_{z\in\overline\Delta(w,r)}G_\Omega(z,w)}.$$
\end{thm}

As a consequence we will obtain the following quantitative version 
of the other implication in the Carleson characterization:

\begin{thm} \label{thm2}
There exists a uniform constant $C>0$ such that for 
$w\in\Omega$, where $\Omega$ is a domain in $\mathbb C$, we have
  $$K_\Omega(w)\leq\frac C{\delta_\Omega(w)^2
      \log\left(1/(\delta_\Omega(w)c_\Omega(w))\right)}.$$
\end{thm}

We will also consider the following counterparts of the Bergman
kernel for higher derivatives for $j=0,1,\ldots$
\begin{align*}
  K^{(j)}_\Omega(w):=\sup\{|f^{(j)}(w)|^2&\colon 
   f\in A^2(\Omega),\ ||f||\leq 1,\\ 
   &f(0)=f'(0)=\dots=f^{(j-1)}(0)=0\}.
\end{align*}
We will prove the following generalization of \eqref{suita}:

\begin{thm} \label{thm3}
For $w\in\Omega\subset\mathbb C$ and $j=0,1,2,\dots$ we have
  $$K^{(j)}_\Omega(w)\geq\frac{j!(j+1)!}\pi(c_\Omega(w))^{2j+2}.$$
\end{thm}

The inequality is optimal, one can easily check that the equality
holds for $\Omega=\Delta$, the unit disc, and $w=0$.

It is clear that the dimension of $A^2(\Omega)$ is infinite
if and only if, for a given $w$, there exists infinitely many
$j$'s such that $K^{(j)}_\Omega(w)>0$. Therefore, Theorem \ref{thm3}
gives a quantitative version of a result of Wiegerinck \cite{W}
who showed that if $\mathbb C\setminus\Omega$ is not polar then
$A^2(\Omega)$ is infinitely dimensional.

Since the proof of Theorem \ref{thm2} also easily gives the upper 
bound
  $$K^{(j)}_\Omega(w)\leq\frac{C_j}{\delta_\Omega(w)^{2+j}
      \log\left(1/(\delta_\Omega(w)c_\Omega(w))\right)},\ \ \ 
      w\in\Omega,$$
and by Proposition \ref{max} below we have the following
characterization of domains in dimension one:

\begin{thm}
For $w\in\Omega\subset\mathbb C$ and $j=0,1,2,\dots$ 
the following are equivalent

i) $\mathbb C\setminus\Omega$ is not polar;

ii) $K_\Omega^{(j)}(w)>0$;

iii) $\log K_\Omega^{(j)}$ is smooth and strongly subharmonic;

iv) $A^2(\Omega)\neq\{0\}$;

v) $\text{\rm dim\,}A^2(\Omega)=\infty$. 
\end{thm}

A different proof of the Suita conjecture \eqref{suita} was 
given in \cite{Blb}. It follows from the lower bound
\begin{equation}\label{blb}
  K_\Omega(w)\geq\frac 1{e^{-2t}
     \lambda(\{G_\Omega(\cdot,w)<t\})}
\end{equation} 
for $t<0$. This inequality was proved in \cite{Blb} using
the tensor power trick which requires a corresponding inequality
for pseudoconvex domains in $\mathbb C^n$ for arbitrary $n$.
As noticed by Lempert (see also \cite{BL}), \eqref{blb} can
also be proved using the variational formula for the Bergman 
kernel in $\mathbb C^2$ of Maitani-Yamaguchi \cite{MY}
(generalized by Berndtsson \cite{Be} to higher dimensions). Both 
proofs therefore make crucial use of several complex variables. 
It would be interesting to find a purely one-dimensional proof 
of \eqref{blb}.

It was shown in \cite{BZ} that the right-hand side of \eqref{blb}
is non-increasing in $t$ (it is an open problem in higher
dimensions). Also, a more general conjecture
was given, namely that the function 
\begin{equation}\label{funkcja}
  (-\infty,0)\ni t\longmapsto\log\lambda(\{G_\Omega(\cdot,w)<t\})
\end{equation}  
is convex. A counterexample was found by Forn\ae ss \cite{F}.
It was also shown numerically in \cite{AC} that the conjecture
does not hold in an annulus. Here we will generalize and simplify 
both results proving the following:

\begin{thm}\label{thm4}
Assume that $w\in\Omega$, where $\Omega$ is a domain in $\mathbb C$,
are such that $\nabla G(z_0)=0$ for some 
$z_0\in\Omega\setminus\{w\}$, where $G=G_\Omega(\cdot,w)$. Then 
the function \eqref{funkcja} is not convex near $t_0=G(z_0)$.
\end{thm}  

Note that for example any regular domain $\Omega$ which is not 
simply connected satisfies the assumption of Theorem \ref{thm4}
for any $w$: it is enough to take maximal $t_0$ such that 
$\{G<t_0\}$ is simply connected. Then there exists $z_0$
such that $\nabla G(z_0)=0$.

\section{Upper bounds for the Bergman kernel}

In this section we will prove Theorems \ref{thm1} and \ref{thm2}.

\begin{proof}[Proof of Theorem \ref{thm1}]
We may assume that $\Omega$ is bounded and smooth, $w=0$, and
$r<\delta_\Omega(0)$. Take $f\in A^2(\Omega)$, without loss
of generality we may take such an $f$ which is defined in a 
neighborhood of $\overline\Omega$. 
Let $u\in C^\infty(\overline\Omega\setminus\Delta_r)$, where 
$\Delta_r:=\Delta(0,r)$, be harmonic in 
$\Omega\setminus\overline\Delta_r$ and such that
$u=1$ on $\partial\Omega$ and $u=0$ on $\partial\Delta_r$. Then
  $$f(0)=\frac 1{2\pi i}\int_{\partial\Omega}\frac{f(z)}z\,dz
    =\frac 1{2\pi i}
     \int_{\partial(\Omega\setminus\overline\Delta_r)}
       \frac{fu}z\,dz
    =\frac 1\pi
     \int_{\Omega\setminus\overline\Delta_r}
       \frac{fu_{\bar z}}z\,d\lambda.$$
Therefore
  $$|f(0)|^2\leq\frac{||f||^2}{\pi^2r^2}       
       \int_{\Omega\setminus\overline\Delta_r}
            |u_{\bar z}|^2d\lambda
       =\frac{||f||^2}{4\pi^2r^2}       
       \int_{\partial\Omega}u_n\,d\sigma,$$
where $u_n$ denotes the outer normal derivative of $u$ at 
$\partial\Omega$. Denoting $G=G_\Omega(\cdot,0)$, we have
  $$\frac G{\displaystyle-\max_{\overline\Delta_r}G}+1\leq u,$$
and therefore on $\partial\Omega$
  $$u_n\leq\frac{G_n}{\displaystyle-\max_{\overline\Delta_r}G}.$$
The required estimate now follows from the fact that
  $$\int_{\partial\Omega}G_nd\sigma=2\pi.$$ 
\end{proof}

\begin{proof}[Proof of Theorem \ref{thm2}]
Denote $G=G_\Omega(\cdot,w)$, $R=\delta_\Omega(w)$ and assume 
that $0<r<R$. Then by the Poisson formula
  $$\max_{\overline\Delta(w,r)}G
     \leq\frac{R-r}{2\pi(R+r)}\int_0^{2\pi}G(w+Re^{it})dt
     =\frac{R-r}{R+r}\log(Rc_\Omega(w)).$$
By Theorem \ref{thm1}
  $$K_\Omega(w)\leq\frac{R+r}{2\pi(R-r)r^2\log(1/(Rc_\Omega(w)))}.$$
We can now take for example $r=R/2$ and the estimate follows.
\end{proof}

The smallest constant the above proof gives will be obtained for
$r=(\sqrt 5-1)R/2$, then
  $$C=\frac{11+5\sqrt 5}{4\pi}.$$

\section{Proof of the lower bound for $K_\Omega^{(j)}$}

In this section we prove Theorem \ref{thm3}. We follow the 
method from \cite{Binv}. We could have also used another 
method from \cite{Blb} but this would require to go to several
complex variables and we prefer to have a purely one-dimensional
argument. 

\begin{proof}[Proof of Theorem \ref{thm3}]
We assume that $w=0$, $\Omega$ is bounded and smooth, and denote
$G=G_\Omega(\cdot,0)$. Set
  $$\alpha:=\frac{\partial}{\partial\bar z}\left(z^j
     \chi(|z|)\right)=\frac{z^{j+1}\chi'(|z|)}{2|z|}$$
and
  $$\varphi:=(2j+2)G+\eta\circ G,\ \ \ \psi:=\gamma\circ G,$$
where $\chi\in C^{0,1}((0,\infty))$, 
$\eta\in C^{1,1}((-\infty,0))$,
$\gamma\in C^{0,1}((-\infty,0))$ will be defined later. We assume
that $\eta$ is convex and nondecreasing (so that $\varphi$
is subharmonic), $(\gamma')^2<\eta''$, and that 
$(\gamma'\circ G)^2\leq\delta\eta''\circ G$ on the support of $\alpha$ 
for some constant $\delta$ with $0<\delta<1$. Then by Theorem 2 in 
\cite{Binv} one can find $u\in L^2_{loc}(\Omega)$ such that
$F:=z^j\chi(|z|)-u$ is holomorphic and
\begin{equation}\label{dbe}
  \int_\Omega|u|^2\,\Gamma\circ G\,d\lambda\leq
    \frac{1+\sqrt\delta}{1-\sqrt\delta}\int_\Omega
    \frac{|\alpha|^2}{\eta''\circ G\,|G_z|^2}\,
      e^{2\psi-\varphi}\,d\lambda,
\end{equation}
where
  $$\Gamma:=\left(1-\frac{(\gamma')^2}{\eta''}\right)
     e^{2\gamma-\eta-(2j+2)t}.$$       
Take $\ep>0$ and assume that $\chi(|z|)\subset\{|z|\leq\ep\}$.
We choose $T=T(\ep)<0$ such that $\{|z|\leq\ep\}\subset\{G\leq T\}$.
Since $|G-\log|z||\leq C_0$ near $0$, we may take $T:=\log\ep+C_0$.

Similarly as in \cite{Binv} for $s<0$ we define
\begin{align*}
  \eta_0(s)&:=-\log(-s+e^s-1),\\
  \gamma_0(s)&:=-\log(-s+e^s-1)+\log(1-e^s),
\end{align*}
so that 
  $$\left(1-\frac{(\gamma_0')^2}{\eta_0''}\right)
     e^{2\gamma_0-\eta_0-s}=1$$   
and
\begin{equation}\label{war}
  \lim_{s\to -\infty}\left(2\gamma_0(s)-\eta_0(s)
     -\log\eta_0'(s)\right)=0.
\end{equation}
We now set
  $$\eta(t):=\begin{cases}\eta_0\big((2j+2)t\big)\ &t\geq T,\\
      -\delta\log(T-t+a)+b\ &t<T,\end{cases}$$
and
  $$\gamma(t):=\begin{cases}\gamma_0\big((2j+2)t\big)\ &t\geq T,\\
      -\delta\log(T-t+a)+\wt b\ &t<T,\end{cases}$$
where $\delta=\delta(\ep)>0$ will be determined later and 
$a,b,\wt b$ are uniquely determined by the conditions 
$\eta\in C^{1,1}$, $\gamma\in C^{0,1}$:
\begin{align*}
  a&=a(\ep)=\displaystyle
           \frac\delta{(2j+2)\eta_0'\big((2j+2)T\big)},\\
  b&=b(\ep)=\eta_0\big((2j+2)T\big)+\delta\log a,\\
  \wt b&=\wt b(\ep)=\gamma_0\big((2j+2)T\big)+\delta\log a.
\end{align*}
We see that if we choose $\delta=\sqrt{-T}$ then 
$\delta(\ep)\to 0$ and $a(\ep)\to\infty$ as $\ep\to 0$.

On $(-\infty,T)$ we have $(\gamma')^2=\delta\eta''$ and
  $$\Gamma=(1-\delta)e^{2\wt b-b}
      \frac{e^{-(2j+2)t}}{(T-t+a)^\delta},$$
so that $|z|^{2j}\Gamma\circ G$ is not locally integrable
near 0. By \eqref{dbe} it implies that 
$F(0)=F'(0)=\dots=F^{(j-1)}(0)=0$ and $F^{(j)}(0)=j!\,\chi(0)$.
One can also check that $\Gamma\geq 1$ on $(-\infty,T)$.

Since $|2G_z-1/z|\leq C_1$ near 0, we have $2|z||G_z|\geq 1-C_1\ep$
on $\{|z|\leq\ep\}$. There we also have
\begin{align*}
  \frac{e^{2\psi-\varphi}}{\eta''\circ G}
          &=\frac{e^{2\wt b-b}}\delta(T-G+a)^{2-\delta}
                               e^{-(2j+2)G}\\
          &\leq\frac{e^{2\wt b-b}}
            \delta(\log\ep-\log|z|+2C_0+a)^{2-\delta}e^{-(2j+2)G}.
\end{align*}
Therefore the right-hand side of \eqref{dbe} can bounded from above 
by
\begin{equation}\label{bd}
     \frac{(1+\sqrt\delta)e^{2\wt b-b}\mathcal A(\ep)}
       {\delta(1-\sqrt\delta)(c_\Omega(0))^{2j+2}}
     \int_{\{|z|\leq\ep\}}(\chi'(|z|))^2
       (\log\ep-\log|z|+2C_0+a)^{2-\delta}d\lambda,
\end{equation}
where $\mathcal A(\ep)\to 1$ as $\ep\to 0$. The optimal
choice for $\chi$ is
  $$\chi(r)=(2C_0+a)^{\delta-1}-(\log\ep-\log r+2C_0+a)^{\delta-1},$$
then \eqref{bd} takes the form
  $$\frac{2\pi(1+\sqrt\delta)^2e^{2\wt b-b}(2C_0+a)^{\delta-1}
    \mathcal A(\ep)}{\delta(c_\Omega(0))^{2j+2}}.$$
Note that
  $$\frac{e^{2\wt b-b}(2C_0+a)^{1-\delta}}\delta=
    \left(\frac{2C_0+a)}a\right)^{1-\delta}
    \frac{e^{2\gamma_0(S)-\eta_0(S)-\log\eta_0'(S)}}{2j+2},$$    
where $S=(2j+2)T$. Combining this with \eqref{war} and the fact 
that $\chi(0)=(2C_0+a)^{\delta-1}$ we will obtain
  $$\liminf_{\ep\to 0}\frac{|F^{(j)}(0)|^2}{||F||^2}
    =\frac{j!(j+1)!}\pi(c_\Omega(w))^{2j+2}.$$
\end{proof}

Using standard methods we will also prove the following formula 
for the Laplacian of $\log K_\Omega^{(j)}$. It is of course
well known for $j=0$ and the proof is essentially the same
in general.

\begin{prop}\label{max}
For a domain $\Omega$ in $\mathbb C$ such that 
$\mathbb C\setminus\Omega$ is not polar and $j=0,1,\dots$ we have
  $$\frac{\partial^2}{\partial z\partial\bar z}
     \left(\log K_\Omega\right)=\frac{K^{(j+1)}_\Omega}
        {K^{(j)}_\Omega}.$$
\end{prop}

\begin{proof}
Denote $H_0=A^2(\Omega)$ and for a fixed $w\in\Omega$ and
$k=1,2,\dots$ set
  $$H_k:=\{f\in A^2(\Omega)\colon f(w)=f'(w)=\dots=f^{(k-1)}(w)=0\},$$
Since $H_k$ is of codimension at most 1 in $H_{k-1}$, we can find 
an orthonormal system $\varphi_0,\varphi_1,\dots$  in
$A^2(\Omega)$ such that $\varphi_k\in H_k$ for all $k$. This
means that $\varphi_l^{(j)}(w)=0$ for $l>j$. For $f\in H_j$ we have
  $$f=\sum_{l\geq j}\langle f,\varphi_l\rangle\,\varphi_l.$$
Therefore
  $$K_\Omega^{(j)}(z)=\sum_{l\geq j}|\varphi_l^{(j)}(z)|^2$$
and 
  $$K_\Omega^{(j)}(w)=|\varphi_j^{(j)}(w)|^2.$$
Since
  $$(\log K)_{z\bar z}=\frac{KK_{z\bar z}-|K_z|^2}{K^2}$$
and 
  $$(K_\Omega^{(j)})_{z\bar z}(w)
     =|\varphi_j^{(j+1)}(w)|^2+|\varphi_{j+1}^{(j+1)}(w)|^2,$$
  $$(K_\Omega^{(j)})_z(w)
    =\varphi_j^{(j+1)}(w)\,\overline{\varphi_j^{(j)}(w)},$$
we will obtain
  $$(\log K_\Omega^{(j)})_{z\bar z}(w)
     =\frac{|\varphi_{j+1}^{(j+1)}(w)|^2}
          {|\varphi_{j}^{(j)}(w)|^2}$$
and the proposition follows.          
\end{proof}

\section{Proof of Theorem \ref{thm4}}

Let $t_j\to t_0$ be a sequence of regular values for $G$.
It will be enough to show that $\gamma'(t_j)\to\infty$, where
$\gamma(t)=\lambda(\{G_\Omega(\cdot,w)<t\})$. By the co-area
formula we have
  $$\gamma(t)=\int_{-\infty}^t\int_{\{G=s\}}
     \frac{d\sigma}{|\nabla G|}\,ds,$$
and therefore 
  $$\gamma'(t_j)=\int_{\{G=t_j\}}\frac{d\sigma}{|\nabla G|}.$$
It is convenient to assume that $z_0=0$. Since $G$ is harmonic in 
$\Omega\setminus\{w\}$, it follows that there exists a holomorphic 
$h$ near $0$ such that $h(0)\neq 0$ and for some $n\geq 2$ we have 
  $$G(z)=t_0+\Re\left(z^nh(z)\right).$$
It follows that near $0$ we have
  $$|\nabla G(z)|\leq C_1|z|^{n-1}.$$  
We can also find a biholomorphic $F$ near 0 such that 
$G(F(\zeta))=t_0+\Re(\zeta^n)$. We then have
  $$|\nabla G(F(\zeta))|\leq C_2|\zeta|^{n-1}$$  
and for some $r>0$
  $$\int\limits_{\{G=t_j\}}\frac{d\sigma}{|\nabla G|}
     \geq\frac 1{C_3}\int\limits_{
        \{\zeta\in\Delta(0,r)\colon\Re(\zeta^n)=t_j-t_0\}}
    \frac{d\sigma}{|\zeta|^{n-1}}\longrightarrow\infty$$
as $j\to\infty$. \qed


\begin{thebibliography}{99}

\bibitem{AC} {\sc P. \AA hag, R. Czy\.z, P.\,H. Lundow},
{\sl A counterexample to a conjecture by B\l ocki-Zwonek},
Experiment. Math. (to appear)

\bibitem{Be} {\sc B. Berndtsson}, {\sl Subharmonicity properties 
of the Bergman kernel and some other functions associated to 
pseudoconvex domains}, Ann. Inst. Fourier 56 (2006), 1633--1662

\bibitem{BL} {\sc B. Berndtsson, L. Lempert}, {\sl A proof 
of the Ohsawa-Takegoshi theorem with sharp estimates},
J. Math. Soc. Japan, 68 (2016), 1461--1472

\bibitem{Binv} {\sc Z. B\l ocki}, {\sl Suita conjecture and the 
Ohsawa-Takegoshi extension theorem}, Invent. Math. 193 (2013),
149-158

\bibitem{Blb} {\sc Z. B\l ocki}, {\sl A lower bound for the Bergman
kernel and the Bourgain-Milman inequality}, Geometric Aspects 
of Functional Analysis, Israel Seminar (GAFA) 2011-2013,
eds. B. Klartag, E. Milman, Lect. Notes in Math. 2116, pp. 53--63, 
Springer, 2014 

\bibitem{BZ} {\sc Z. B\l ocki, W. Zwonek}, {\sl Estimates for 
the Bergman kernel and the multidimensional Suita conjecture}, 
New York J. Math. 21 (2015) 151--161 

\bibitem{C} {\sc L. Carleson}, {\sl Selected Problems on 
Exceptional Sets}, Van Nostrand, 1967

\bibitem{Co} {\sc J.\,B. Conway}, {\sl Functions of One Complex
Variable II}, Springer, 1995

\bibitem{DF} {\sc H. Donnelly, C. Fefferman}, {\sl $L^2$-cohomology and
index theorem for the Bergman metric}, Ann. of Math. 118 (1983),
593--618

\bibitem{F} {\sc J.\,E. Forn\ae ss}, {\sl On a conjecture by 
Blocki and Zwonek}, Sci. China Math. (to appear)

\bibitem{MY} {\sc F. Maitani, H. Yamaguchi}, {\sl Variation 
of Bergman metrics on Riemann surfaces}, Math. Ann, 330 (2004), 
477--489

\bibitem{S} {\sc N. Suita}, {\sl Capacities and kernels on Riemann
surfaces}, Arch. Ration. Mech. Anal. 46 (1972), 212--217

\bibitem{W} {\sc J. Wiegerinck}, {\sl Domains with finite
dimensional Bergman space}, M. Zeit. 187 (1984), 559--562

\end{thebibliography}
\end{document}